\documentclass[12pt]{amsart}
\usepackage{amsmath,amssymb,amsfonts,amsthm,amstext,graphicx,color,enumerate}
\usepackage{cleveref,url,mathtools,mathdots,verbatim}
\usepackage{cite}

\usepackage{epstopdf}

\usepackage[margin=3.5cm]{geometry}

\newtheorem{thm}{Theorem}
\newtheorem*{thm*}{Theorem}
\newtheorem{prop}[thm]{Proposition}
\newtheorem{lemma}[thm]{Lemma}
\newtheorem{cor}[thm]{Corollary}

\theoremstyle{definition}
\newtheorem*{defn}{Definition}

\crefname{thm}{Theorem}{Theorems}
\crefname{lemma}{Lemma}{Lemmas}
\crefname{prop}{Proposition}{Propositions}
\crefname{cor}{Corollary}{Corollaries}
\crefname{section}{Section}{Sections}
\crefname{figure}{Figure}{Figures}
\crefname{question}{Question}{Questions}
\crefname{equation}{}{}
\crefname{section}{Section}{Sections}

\newcommand{\E}{\mathbb E}
\renewcommand{\P}{\mathbb P}
\newcommand{\N}{\mathbb N}
\newcommand{\Z}{\mathbb Z}

\newcommand{\hS}{\widehat S}
\newcommand{\hB}{\widehat B}
\newcommand{\hp}{\widehat p}
\newcommand{\hq}{\widehat q}
\newcommand{\hqp}{\widehat q{\,}'}

\newcommand{\hx}{\widehat x}
\newcommand{\cell}{{\text{\tt cell}}}
\newcommand{\boxright}{\rhd}

\newcommand{\lev}{\Lambda}

\newcommand{\df}[1]{\textbf{\boldmath #1}}

\title[Polluted Bootstrap Percolation]{Polluted Bootstrap Percolation \\ with Threshold Two  in All Dimensions}
\keywords{bootstrap percolation; cellular automaton; critical scaling}
\subjclass[2010]{60K35; 82B43}

\date{3 April 2017}

\author{Janko Gravner}
\address{Janko Gravner, Mathematics Dept., University of California, Davis, CA 95616}
\email{gravner@math.ucdavis.edu}
\author{Alexander E.~Holroyd}
\address{Alexander E.~Holroyd,
Microsoft Research, Redmond, WA 98052}
\email{holroyd@microsoft.com}

\begin{document}

\begin{abstract}
In the polluted bootstrap percolation model, the vertices of a graph
are independently declared initially occupied with probability $p$ or closed with probability $q$.  At subsequent steps, a vertex becomes occupied if it is not closed and it has at least $r$ occupied neighbors.  On the cubic lattice $\Z^d$ of dimension $d\geq 3$ with threshold $r=2$, we prove that the final density of occupied sites converges to $1$ as $p$ and $q$ both approach $0$, regardless of their relative scaling.  Our result partially resolves a conjecture of Morris, and contrasts with the $d=2$ case, where Gravner and McDonald proved that the critical parameter is $q/{p^2}$.
\end{abstract}

\maketitle

\section{Introduction}\label{sec-intro}

Bootstrap percolation is a fundamental cellular
automaton model for nucleation and growth from sparse
random initial seeds.  In this article we address how the
model is affected by the presence of pollution in the form
of sparse random permanent
obstacles.

Let $\Z^d$ be the set of $d$-vectors of integers, which we
call \df{sites}, and let $p,q\in[0,1]$ be parameters. In the
\df{initial} (time zero) configuration, each site is chosen
to have exactly one of three possible states:
$$\begin{cases}
\begin{array}{ll}
  \text{\df{closed}}  & \text{with probability }q; \\
  \text{\df{open} and \df{initially occupied}} & \text{with probability }p; \\
  \text{\df{open} but not initially occupied} & \text{with probability }1-p-q.
\end{array}
\end{cases}
$$
Initial states are chosen independently for different sites.  Closed sites represent pollution or obstacles, while occupied sites represent a growing agent.

The configuration evolves in discrete time steps $t=0,1,2,\ldots$ as follows.
As usual we make $\Z^d$ into a graph by declaring sites
$u,v\in\Z^d$ to be neighbors if $\|u-v\|_1=1$. The
\df{threshold} $r$ is an integer parameter.  An open site $x$ that is unoccupied at time $t$
becomes occupied at time $t+1$ if and only if
\begin{equation}
\text{at least $r$ neighbors of $x$ are occupied}
\label{standard}
\end{equation}
 at time $t$.
Closed sites remain closed forever and
cannot become occupied.  Open sites remain open.  Once a site is occupied, it remains
occupied.  In the main cases of interest, $d\geq r\geq 2$.


Bootstrap percolation without pollution (the case $q=0$ in our formulation) has a long and rich history with many surprises.
For $d\geq r \geq 1$, there is no phase transition in $p$, in the sense that every site of $\Z^d$ is eventually occupied almost surely for every $p>0$, as proved in \cite{van-enter} ($d=2$) and \cite{schonmann} ($d\geq 3$).  The metastability properties of the model on finite regions are understood in great depth (see e.g.\ \cite{AL,Hol1,BBDM,GHM}), while a broad range of variant growth rules have also been explored (e.g.\ \cite{GG,DE,BDMS}).  For further background see the discussion later in the introduction, and the excellent recent survey \cite{Mor}.

The polluted bootstrap model (i.e.\ the case $q>0$) was introduced by Gravner and McDonald \cite{GM} in 1997.  The principal quantity of interest is the {\it final density\/} of occupied sites, i.e.\ the probability that the origin is eventually occupied, in the regime where $p$ and $q$ are both small.  In dimension $d=2$ with threshold $r=2$,  Gravner and McDonald proved that the final density is strongly dependent
on the relative scaling of $p$ and $q$.  Specifically,
there exist constants $c,C>0$ such that, as $p\to 0$ and $q\to 0$ simultaneously,
$$
\P\bigl(\text{the origin
is eventually occupied}\bigr)\to
\begin{cases}
  1, & \text{if } q<c p^2;\\
  0, & \text{if } q>C p^2.
\end{cases}
$$

In this article we give the first rigorous treatment of the polluted bootstrap percolation model in dimensions $d\geq 3$.  We take the threshold $r$ to be $2$.  (Threshold $r=3$ is addressed in a companion paper \cite{GHS} by the current authors together with Sivakoff, as discussed below).  Our main result is that, in contrast with dimension $d=2$, occupation prevails regardless of the $p$ versus $q$ scaling.

\begin{thm}\label{main}
 Consider polluted bootstrap percolation on $\Z^d$ with $d\ge 3$, threshold
  $r=2$, density $p>0$ of initially occupied sites,
  and density $q>0$ of closed sites.  We have
$$
\P\bigl(\text{\rm the origin
is eventually occupied}\bigr)\to 1\qquad\text{as }(p,q)\to (0,0).
$$
Moreover, the probability that the origin lies in an infinite
connected set of eventually occupied sites also tends to $1$.
The same statements hold for modified bootstrap percolation.
\end{thm}

In the above statement, a set of sites is called \df{connected} if it induces a connected subgraph of $\Z^d$.  The \df{modified} bootstrap percolation model is a well-known variant of the standard model, in which the condition \eqref{standard} for a site to become occupied is replaced with:
\begin{equation*}
\text{for at least $r$ of the directions $i=1,\ldots, d$,
either $x-e_i$ or $x+e_i$ is occupied,}
\end{equation*}
where $e_i$ is the $i$th coordinate vector.  (As before, closed sites cannot become occupied, and occupied sites remain occupied forever).

\cref{main} resolves Conjecture 4.6 of Morris \cite{Mor} in the key
case $r=2$.  To be precise, this conjecture may be expressed as: for all $d>r \geq 1$, there exists an infinite connected eventually occupied set with probability at least $1/2$ for $(p,q)$ sufficiently close to $(0,0)$.  The author states that the conjecture seems to be very difficult.

Defining
$\phi(p,q)=\phi_{d,r}(p,q)$ to be the probability that the origin is eventually occupied, it follows from the obvious monotonicities of the model that $\phi$ is (weakly) increasing in $p$ and decreasing in $q$.  Therefore, the convergence in \cref{main} is equivalent to  $\lim_{q\to 0} \lim_{p\to 0} \,\phi(p,q)=1$.  This formulation will be reflected in our proof.  We will show that for $q$ sufficiently small there is an infinite structure of open sites on which occupation can spread, no matter how small $p$, and that the density of this structure tends to $1$ as $q\to 0$.  Our methods are very different from those in previous works on bootstrap percolation, and involve the technology of oriented surfaces introduced recently in \cite{DDGHS}.

Our result reveals an interesting phase transition.  Let $r=2$ and $d\geq 3$ and consider the decreasing function $\phi^+(q):=\phi(0^+,q)=\lim_{p\to 0^+} \phi(p,q)$.  \cref{main} implies that $\phi^+(q)>0$ for $q$ sufficiently close to $0$.  On the other hand, standard arguments imply that $\phi^+(q)=0$ if $q$ exceeds one minus the critical probability $p_c^{\textrm{site}}(\Z^d)$ of site percolation.  Therefore the critical probability
$$q_c:=\inf\{q: \phi^+(q)=0\}$$
is nontrivial.  In fact, we show the following slightly stronger fact involving a strict inequality.

\begin{cor} \label{main-follow} Consider the setting of
\cref{main}. The critical value $q_c$ defined above satisfies
$0<q_c\le 1-p_c^{\text{\rm site}}(\Z^d)$.  For $d=3$, the latter inequality is strict.  The function $\phi^+$ vanishes
on $(q_c,1]$, is strictly positive
on $[0,q_c)$, and converges to $1$ as $q\to 0$.
\end{cor}

Our methods do not produce a good lower bound on $q_c$, and give
no information on the behavior of $\phi^+$ near $q_c$.

As mentioned earlier, the companion paper \cite{GHS} treats polluted bootstrap percolation with threshold $r=3$.  The strongest result of \cite{GHS} is for the modified bootstrap percolation model with $d=r=3$.  Similarly to the case $d=r=2$ of \cite{GM}, but in contrast with the $d>r=2$ case of \cref{main}, the final density here depends on the $p$ versus $q$ scaling, but now with a \emph{cube} law (modulo logarithmic factors).  Specifically, as $p,q\to 0$, the final occupied density converges to $1$ if $q< c\,(p/\log p^{-1})^3$, and to $0$ if $q> C p^3$.   Interestingly, the first of these bounds relies crucially on \cref{main} of the current article (together with a straightforward renormalization argument).  The second bound (which is far from straightforward) again uses oriented surfaces, but in a completely different way: to block growth rather than to facilitate it.

We record some simple observations about other choices of the threshold $r$.  For $d=r$, notwithstanding the detailed results of \cite{GM,GHS}, an easy argument rules out the conclusion $\lim_{(p,q)\to(0,0)} \phi(p,q)=1$ of \cref{main}.  Indeed, if $p,q\to 0$ with $p=o(q^{2^d})$ then with high probability there exist $M<0<N$ such that no site in the box $\{0,1\}^{d-1}\times [M,N]$ is initially occupied but every site on the two ends $\{0,1\}^{d-1}\times \{M,N\}$ is closed.  On this event, the origin cannot become occupied.  (For the modified model, the same argument works even for the line $\{0\}^{d-1}\times [M,N]$, giving the same conclusion under the weaker assumption $p=o(q)$.  Similar comparisons involving
$\{0,1\}^{d-d'}\times \Z^{d'}$ or $\{0\}^{d-d'}\times \Z^{d'}$ for $d>d'$ are available, which, when combined with the results of \cite{GM} for $d'=2$ or \cite{GHS} for $d'=3$, yield further improvements.)  On the other hand, the case of threshold $r=1$ is easily understood via standard site percolation: the final occupied set is simply the union of all open clusters that contain initially occupied sites.  (This observation is relevant to \cref{main-follow}.)  Finally, thresholds $r>d$ are less interesting to us, since, even with no closed sites, there are finite sets such as $\{0,1\}^d$ that remain unoccupied forever if unoccupied initially, so $\lim_{(p,q)\to(0,0)} \phi(p,q)=0$.

\subsection*{Background}

Bootstrap percolation is an established model for nucleation and
metastability, and one of very few cellular automaton models with a well-developed mathematical theory.  It has been applied in physics, biology, and
social science to various growth phenomena, including crack formation, crystal growth, and spread of information or infection.  See \cite{GZH} for a recent example.  Bootstrap percolation has been used in the rigorous analysis of other models such as sandpile and Ising models; see e.g.\ \cite{Mor}.  The evolving set method in Markov mixing theory can be viewed as bootstrap percolation with a randomly varying threshold \cite{evolve}.

Bootstrap percolation was first considered on trees \cite{CLR}, but the lattice $\Z^d$ with its physics connotations has received the most attention.  There has been recent interest in mean-field and power-law graphs, motivated in part by applications to social networks; see e.g.\ \cite{JLTV,amini2,koch}.

Polluted bootstrap percolation was introduced in \cite{GM} on the two dimensional lattice.  Potential areas of application include the effects of impurities on crystal growth, of immunization on epidemics, or of interventions on spread of rumors.
Since \cite{GM}, rigorous progress on growth processes in random environments has been limited, and the case of polluted bootstrap percolation in three and higher dimensions has been entirely open until now.
Here are some examples of work on related models.
Investigation of asymptotic shapes in models related
to polluted bootstrap percolation with $r=1$ was
initiated in \cite{GMa}; a recent paper \cite{JLTV} studies
such processes on a complete graph with excluded edges;
and \cite{DEKNS} addresses a Glauber dynamics (which can be viewed
as a non-monotone version of bootstrap percolation) with
``frozen'' vertices. Polluted bootstrap percolation
and closely related models have been used in empirical studies
of complex networks with ``damaged'' vertices \cite{BDGM2, BDGM1}.

A key element in our proof will be the simple but powerful method of random oriented surfaces recently introduced in \cite{DDGHS}.  This method has been further used and developed in a variety of contexts \cite{interface,ent,geom,embed,comb,bod-tex}, but ours is the first application to cellular automata so far as we are aware.  A distinct application to polluted bootstrap percolation will appear in \cite{GHS}.

Another useful tool will be the results of \cite{LSS}
concerning domination of finitely dependent processes.  A
random configuration $X=(X_v)_{v\in \Z^d}$ taking values in
$\{0,1\}^{\Z^d}$ is called \df{$m$-dependent} if
$(X_v)_{v\in A}$ and $(X_v)_{v\in B}$ are independent of
each other whenever the sets $A$ and $B$ are at distance
greater than $m$.  The relevant result of \cite{LSS} is
that for any $p<1$ there exists $p'=p'(p,d,m)<1$ such that,
if $X$ is $m$-dependent and satisfies $\E X_v\geq p'$ for
all $v$, then $X$ stochastically dominates an i.i.d.\ process
with parameter $p$.

\subsection*{Outline of proof and organization}

The modified bootstrap percolation model is ``weaker'' than the standard model, in the sense that it is more difficult for a site to become occupied, so that for a given initial configuration, the occupied set for the modified model is a subset of that for the standard model at each time $t$.  Therefore, it suffices to prove the conclusions of \cref{main} for the modified model.  Moreover, we may without loss of generality assume that $d=3$.  Indeed, for $d\geq 4$ we may restrict to the $3$-dimensional subspace $\Z^3\times\{0\}^{d-3}$.  Any site that becomes occupied in the $d=3$ model restricted to the subspace also becomes occupied in the full model on $\Z^d$ (where in both cases $r=2$).  Therefore, for the remainder of the paper we consider the modified bootstrap percolation model with $r=2$ on $\Z^3$ except where explicitly stated otherwise.

In the absence of closed sites, the two-dimensional bootstrap rule
fills $\Z^2$ from any positive density $p$ of occupied sites.
This suggests the following approach.  For $q$ sufficiently small we may attempt to construct an infinite two-dimensional surface that avoids closed sites and behaves like $\Z^2$, in the sense that it also admits growth by the $r=2$ model for any $p>0$.  In \cref{sec-open-curtains} we indeed construct an oriented surface, called a \emph{curtain}, with some of the required properties.  In particular, starting from an infinite fully occupied half space of $\Z^3$, a curtain will become fully occupied almost surely for any $p>0$.  The construction of the curtain itself does not involve $p$, and does not depend on the locations of initially occupied sites.

A curtain alone is not sufficient to prove \cref{main}, because a \emph{finite} occupied nucleus does not lead to indefinite growth on a curtain.
To address this, we will use a renormalization argument involving curtains with different orientations that intersect each other.  This part of the argument \emph{will} involve $p$, in the determination of a length scale. In \cref{sec-sails} we construct the unit of our renormalization, which is a curtain restricted to a finite box, with carefully constrained geometry, and scaled to facilitate the required intersections.  This modified curtain is called a \emph{sail}.  The size of the box is chosen to be a power of $p^{-1}$, which allows the sail to contain sufficient initially occupied sites for growth similar to that on a curtain.  In \cref{sec-activation} we use comparison methods to show that if two sails intersect appropriately then occupation is transmitted from one to the other.  Finally, \cref{sec-renormalization} completes the renormalization argument, which involves comparison of an infinite network of sails with supercritical oriented percolation, together with ``sprinkling'' for the initial nucleation.

We conclude the paper with a list of open problems.

\subsection*{Notation and conventions}

As stated earlier, we work with the polluted modified bootstrap percolation model with threshold $r=2$ on $\Z^3$ unless stated otherwise.  The cubic lattice, also denoted $\Z^3$, is the graph with vertex set $\Z^3$ and with an edge between sites $u$ and $v$ whenever $\|u-v\|_1=1$.  When discussing sets of sites, connectivity and components always refer to this graph.

When describing subsets of $\Z^3$, intervals will be understood to denote their intersections with $\Z$, so $[a,b)$ denotes $[a,b)\cap\Z=\{a,a+1,\ldots,b-1\}$, etc.  Let $\N$ be the set of nonnegative integers.
We will frequently wish to consider $2$-dimensional layers of $\Z^3$, which by convention will be taken perpendicular to the $3$rd coordinate.  Thus, for $k\in\Z$ we define the $k$th \df{layer} to be
$$\lev_k:=\Z^2 \times \{k\}=\bigl\{x\in\Z^3:x_3=k\bigr\}.$$
Let $\langle \cdot,\cdot\rangle$ denote the standard inner product on $\Z^3$, and let $e_1,e_2,e_3$ be the standard coordinate vectors.

We will consider paths of various types, not always with nearest-neighbor steps.  In general, a \df{path} is a finite or infinite sequence of sites $(\ldots,)x_0,x_1,\ldots,x_n(,\ldots)$.  Its \df{steps} are the vectors $(\ldots,)x_1-x_0,x_2-x_1,\ldots,x_n-x_{n-1}(,\ldots)$.  It is a nearest-neighbor path if all steps are of the form $\pm e_i$.  It is self-avoiding if all its sites are distinct.

\section{Curtains}\label{sec-open-curtains}

In this section we introduce the oriented surfaces
underlying our construction in their pure form.
Later they will be modified by scaling and restricting to finite boxes.

\begin{defn} A \df{curtain} is a set $D\subset \Z^3$ satisfying
the following.
\begin{enumerate}
\item[(C1)] For any $k\in \Z$, the intersection
    $D\cap\lev_k$ with layer $k$ is an infinite path
    comprising steps $e_1$ and $-e_2$, with no three
    consecutive steps in the same direction; i.e.\ no
    $e_1, e_1, e_1$ or $-e_2, -e_2, -e_2$.
\item[(C2)] For all $x\in D$, either $x+(0,0,-1)\in D$ or
    $x+(1,1,-1)\in D$.
\end{enumerate}
\end{defn}

\cref{fig-proto} in the next section shows the intersection
of a curtain with a box.  The main goal of this section is
to construct an infinite open curtain when $q$ is
sufficiently small.  This will be done adapting the duality
technique introduced in \cite{DDGHS} for construction of
Lipshitz surfaces.  The curtain will form the outer
boundary of a set reachable by certain paths from a fixed
half space. Before giving the construction, we illustrate
the relevance of curtains to bootstrap percolation with the
following lemma. (Formally, the lemma will not be used in
the proof of \cref{main}. Instead we will use a more
specialized variant, \cref{growth-sail}.)

\begin{samepage}
\begin{lemma} \label{curtain-spread}
Let $D$ be a curtain.  Suppose that for every $x\in D$, the
three sites $x$ and $x+(0,0,1)$ and $x+(-1,-1,1)$ are all
open. Moreover, suppose that for every $k\in\N$, the set $
(D\cap\lev_k)+e_3$ contains some initially occupied site.
If $D\cap\lev_0$ is initially entirely occupied, then
 $D\cap \bigcup_{k\in \N} \lev_k$ becomes entirely occupied in the modified bootstrap
 model on $\Z^3$.
\end{lemma}
\end{samepage}

\begin{figure}
\centering
\includegraphics[width=.4\textwidth]{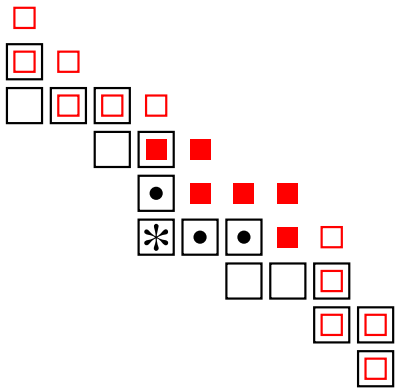}
\caption{An illustration of the proof of \cref{curtain-spread}.
  Two consecutive layers of a curtain are shown from above.  Large squares
are present in the upper layer, and small (red) squares in the lower layer.
The argument showing that the upper layer site marked with a star becomes occupied is indicated.
We consider the portion of the lower layer path shown by filled squares, and deduce that
all the upper layer sites marked with discs become occupied.
  }\label{paths}
\end{figure}

\begin{proof}
By induction on the layer, it suffices to prove that
$D\cap\lev_1$ becomes entirely occupied. This verification
is given in two steps below, and is illustrated in
\cref{paths}. Let $Y:=(D\cap\lev_0)+e_3$ be the set above
the intersection with the bottom layer.

First we claim that every site in
$Y$ eventually becomes occupied.
Indeed,
$Y$ is connected, open, and
contains an occupied site, and
every $y\in Y$ has an occupied neighbor $y-e_3\notin Y$.
The claim therefore follows from the bootstrap rule.

We now claim that every site in $Y-(1,1,0)$ also eventually
becomes occupied. Indeed, consider such a site
$z=y-(1,1,0)$ where $y\in Y$. Since $Y$ is a path with the
properties given in (C1), there exist sites $z+a e_1$ and
$z+b e_2$ in $Y$, where $a,b\in[0,3]$.  Moreover, the
intervening sites $z+i e_1$ and $z+j e_2$ for $i\in(0,a)$
and $j\in(0,b)$ are open by (C2), and each has a neighbor
in $Y$ distinct from $z+a e_1$ and
$z+b e_2$.  Since all sites in $Y$ become occupied, so do all
these sites, whence so does $z$.

The proof is now concluded by observing that
$D\cap\lev_1\subseteq Y\cup(Y-(1,1,0))$.
\end{proof}

Now we proceed with the construction of a curtain.
A \df{permissible path} is a finite sequence of sites $x_0,\ldots,x_n\in\Z^3$ such that every step $x_{i+1}-x_i$ satisfies the following.
Either it is a \df{taxed} step, which is to say that $x_{i+1}$ is \emph{closed}, and $x_{i+1}-x_i$ equals
$$(1,1,0).$$
Otherwise, the step is \df{free}, that is, $x_{i+1}-x_i$ lies in
$$\Bigl\{(-1,0,0),(0,-1,0),(0,0,-1),(-2,1,0),(1,-2,0),(-1,-1,1)\Bigr\}.$$
(with no restriction on the states of sites).

Fix any (deterministic) set $H\subset\Z^3$ and let $A$ be
the (random) set reachable by permissible paths from $H$.
Then define the following outer boundary:
\begin{equation}\label{boundary}
D:=\bigl\{x\notin A: x-(1,1,0)\in A\bigr\}.
\end{equation}

\begin{lemma} \label{curtain} For any choice of $H$, the set $D$
is either empty or an open curtain.
\end{lemma}

The lemma is of course only useful when $D$ is nonempty.
This will be proved to hold under suitable circumstances in
\cref{surface} below.

\begin{proof}[Proof of \cref{curtain}] We must prove that if $D$ is nonempty then it is open and has properties
(C1) and (C2). Consider any $x\in D$. By translation invariance of the definition, we assume
without loss of generality that $x$ is
the origin $0=(0,0,0)$.

Clearly, $0$ is open,
since otherwise the taxed step from $(-1,-1,0)\in A$ would make
$0\in A$.

Turning to property (C1), we have $(-1,-1,0)\in A$
but $0\notin A$, so using the definition of free steps, $(-1,-2,0)\in A$ but
$(1,0,0)\notin A$.
We claim that either
$(1,0,0)\in D$ or $(0,-1,0)\in D$, but not both.
Indeed, if $(0,-1,0)\in A$ then $(1,0,0)\in D$, while
if $(0,-1,0)\notin A$ then $(0,-1,0)\in D$.
A similar argument shows that either $(-1,0,0)\in D$ or $(0,1,0)\in D$ but not both.
This shows that $D\cap \lev_0$ is a union of disjoint paths with steps $e_1$ and $-e_2$.  To check the restriction on three consecutive steps, note that $(0,-3,0)\in A$ but $(2,-1,0)\notin A$, which implies $(0,-3,0),(3,0,0)\notin D$.
To show that there is only one path, note that $(-1,-1,0)$ is a sum of two free steps, so the diagonal $\{(k,k,0):k\in\Z\}$ is partitioned into an interval belonging to $A$ and an interval belonging to $A^C$.  If $D$ is nonempty then both intervals are nonempty, and so the diagonal contains exactly one site in $D$.


To prove property (C2),
note that $(-1,-1,-1)\in A$ but $(1,1,-1)\notin A$. Consequently, if
$(0,0,-1)\notin A$, then $(0,0,-1)\in D$. On the other hand,
if $(0,0,-1)\in A$, then $(1,1,-1)\in D$.
\end{proof}

We now choose $H$ to be the half-space
$$H:=\{x: x_1+x_2+x_3\le 0\},$$
and let $A$ and $D$ by defined as above. Note that, by
property (C1), a curtain intersects the line
$\{(t,t,0):t\in\Z\}$ in exactly one site.

\begin{prop}\label{surface} There exist positive constants $q_0$ and $c$
such that the following holds. For $q<q_0$, the set $D$
constructed above is, almost surely, an open curtain.
Furthermore, the probability that $(1,1,0)\in D$ tends to
$1$ as $q\to 0$, while for any $q<q_0$, the probability
that $D$ intersects the ray $\{(t,t,0): t>k\}$ is less than
$e^{-ck}$ for all $k>0$.
\end{prop}

\begin{proof}
Let $h=(1,1,1)$.  Observe that the scalar product
$\langle x, h\rangle$: equals $2$ when $x$ is the
taxed step; equals $-1$ when $x$ is a free step;
and is nonpositive when $x\in H$.

Fix $k\ge 1$ and suppose $(k,k,0)\in A$.
Then there exists a permissible from $H$ to $(k,k,0)$. By erasing loops, we
may assume that the path is self-avoiding.
Let $n_F$ and $n_T$ be
the number of free and taxed steps
of the path, respectively, and let $n=n_F+n_T$ be the total length. As $\langle (k,k,0),h\rangle=2k$, the above observations about scalar products imply
$-n_F+2n_T\ge 2k$. It follows that $n\ge n_T\ge k$
and $n_T\ge (2k+n)/3$. Therefore,
\begin{equation}\label{surface-eq1}
\begin{aligned}
\P\bigl((k,k,0)\in A\bigr)&\le
\P(\text{there exists a path as above})\\
&\le \sum_{n\ge k}7^n q^{(2k+n)/3}\\
&= (7q)^k\cdot\sum_{n\ge 0}(7 q^{1/3})^{n}\\
&\le 8\cdot (7q)^k, 
\end{aligned}
\end{equation}
provided $q<q_0:=8^{-3}$.

Note that $0\in A$, so that if $(k,k,0)\notin A$ then
$\{(t,t,0):t\in[1,k]\}$ intersects $D$ while
$\{(t,t,0):t>k\}$ does not.  Thus, if $q<q_0$ then
\cref{curtain} implies that $D$ is almost surely nonempty,
and thus is an open curtain by \cref{curtain}.  Equation
\cref{surface-eq1} also gives the claimed exponential
bound, since $8\cdot(7q)^k\leq (56 q_0)^k$. Taking $k=1$ in
\cref{surface-eq1}, we get $\P((1,1,0)\notin D)\le 56q$,
giving the second claim.
\end{proof}

The results from this section are already strongly
suggestive of the conclusions of \cref{main}, although by
no means sufficient to prove them. Indeed, consider the
initial configuration consisting of the fully occupied
half-space $\{x: \langle x, e_3\rangle \le 0\}$, and
elsewhere product measure with densities $p$ and $q$ as
usual. It follows easily from
\cref{curtain-spread,curtain,surface} that the probability
that any fixed site $x\in \Z^3$ is eventually occupied
converges to $1$ as $(p,q)\to (0,0)$.   Indeed, with high
probability $x$ lies in a curtain that has the properties
in \cref{curtain-spread}: the presence of the appropriately placed
open sites can be guaranteed via \cite{LSS}, while the presence of an occupied site in
each layer of the curtain holds almost surely. The
remaining difficulty in the proof of \cref{main} is the
need to replace the occupied half-space with a finite
nucleus.

\section{Sails}\label{sec-sails}

A \df{box} in $\Z^3$ is a Cartesian product of any three
integer intervals. Its \df{dimensions} are the
cardinalities of the three intervals (in order). An
\df{oriented box} is a box with a distinguished corner.

Fix an integer length scale $L$. This scale will be later
chosen to be a suitable function of $p$. A \df{brick} is an
oriented box of dimensions $4L$, $16L$, and $32L$, in any
order.  Bricks will be the units of our renormalization. We
will formulate the required properties of bricks by
translating and scaling a smaller box. The \df{proto-brick}
$\hB$ is the oriented box $[0,4L)\times [0,4L)\times
[0,2L)$ with the distinguished corner at the origin.

We now formulate the key definition in our renormalization
argument. The idea is that the proto-brick contains a
suitably placed portion of a curtain, with properties
analogous to those in \cref{curtain-spread}, but restricted
to the proto-brick.  See \cref{fig-proto} for an
illustration.

\begin{defn}
 The proto-brick $\hB$ is \df{good} if
there exists a set $\hS\subseteq \hB$ with the following
properties:
\begin{enumerate}
\item[(G1)] all sites in the following set are open:
$$
\sigma(\hS):=\bigl\{x,\;x+(0,0,1),\;x+(-1,-1,1): x\in \hS\bigr\}\cap \hB;$$
%
\item[(G2)] $\hS$ satisfies (C2) in the definition of a
    curtain except at the bottom layer: for all $x\in
    \hS\setminus\lev_0$, either $x+(0,0,-1)\in \hS$ or
    $x+(1,1,-1)\in \hS$;
\item[(G3)] $\hS\subseteq\{x: 3L< x_1+x_2+x_3<4L\}$;
\item[(G4)] for each layer $k\in[0,2L)$, the intersection
    $\hS\cap\lev_k$ is an oriented path that starts on
    $\{x:x_1=0\}$, ends on $\{x:x_2=0\}$ and makes steps
    $-e_2$ or $e_1$ with no consecutive three steps of
    the same type; and
\item[(G5)] for each layer except the top, there is an
    occupied site immediately above its intersection with
    $\hS$, i.e.\ $(\hS\cap \lev_k)+e_3$ contains an
    occupied site for each $k\in [0,2L-1)$.
\end{enumerate}
\end{defn}

\begin{figure}
\centering
\includegraphics[width=.5\textwidth]{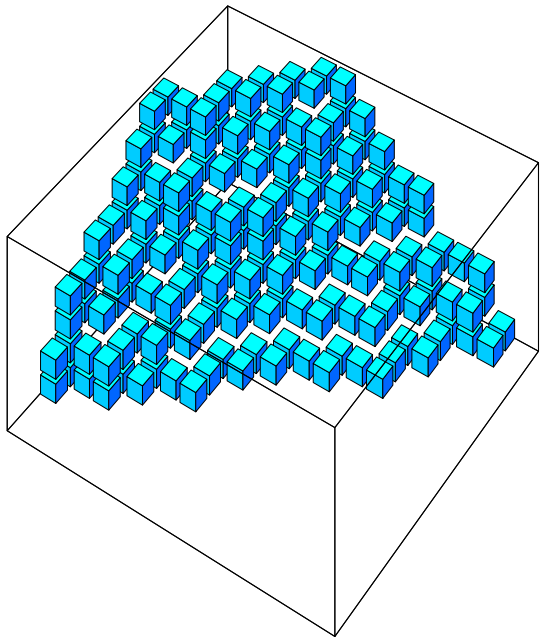}
\caption{\label{fig-proto}  A good proto-brick with $L=4$.
The set $\hS$ comprises the (sites in the centers of) colored cubes.
The origin is at the back corner, hidden by the
set.  The third coordinate axis is vertical.}
\end{figure}

Next we scale up this definition to a brick, starting with
one in a standard location and orientation.  Let $B$ be the
brick $[0,4L)\times [0,16L)\times [0,32L)$ with the
distinguished corner at the origin. For $x\in \hB$, define
the following subset of $B$:
$$
\cell(x)=(x_1,4x_2,16x_3)+\{0\}\times[0,4)\times [0,16).
$$
See \cref{fig-cell}. For a given configuration on $B$, we
define an \df{auxiliary configuration} on $\hB$ by
declaring a site $x\in \hB$ open if all sites in $\cell(x)$
are open; otherwise, we declare $x$ closed. We also call
$x$ initially occupied if all sites in $\cell(x)$ are
initially occupied. We call $B$ \df{good} if, in the
auxiliary configuration, $\hB$ is good.  See
\cref{fig-cell}.

If $B$ is good and $\hS$ is any set satisfying the above
conditions, then we call
$$
S=\bigcup_{x\in \hS} \cell(x)
$$
a \df{sail} for $B$. Thus $B$ is good if and only it has a
sail.

Define the \df{tail} and \df{head} of $B$ to be its lower
and upper halves, $[0,4L)\times [0,16L)\times [0,16L)$ and
$[0,4L)\times [0,16L)\times [16L,32L)$ respectively. The
\df{tip} of $B$ is the box $[0,4L)\times [0,4L)\times
[16L,32L)$, which is a quarter of the head. See
\cref{fig-cell}. The \df{base} of $B$ is the bottom layer
of cells
$
\bigcup_{x\in \hB\cap\lev_0} \cell(x).
$
If $B$ is good and $S$ is a sail for $B$, then the
\df{head}, \df{tail}, \df{base}, and \df{tip} of $S$ are
the intersections of $S$ with the corresponding subsets of
$B$.

If the brick $B$ is good, and $S$ is a sail for $B$, then
we say that $S$ is \df{activated} by time $t$ if every site
in the the head of $S$ is occupied at time $t$.

Now we transfer all the above definitions to an arbitrary
brick $B'$ by isometry.  More precisely, let $\eta$ be an
isometry of $\Z^3$ that maps $B$ to $B'$, respecting the
distinguished corners.  The head of $B'$ is the image under
$\eta$ of the head of $B$. The brick $B'$ is good if
applying $\eta^{-1}$ to the configuration makes $B$ good,
in which case a sail for $B'$ is an image under $\eta$ of a
sail for $B$ in that configuration, and so on.

\begin{figure}
\centering
{}\hfill
\includegraphics[width=.1\textwidth]{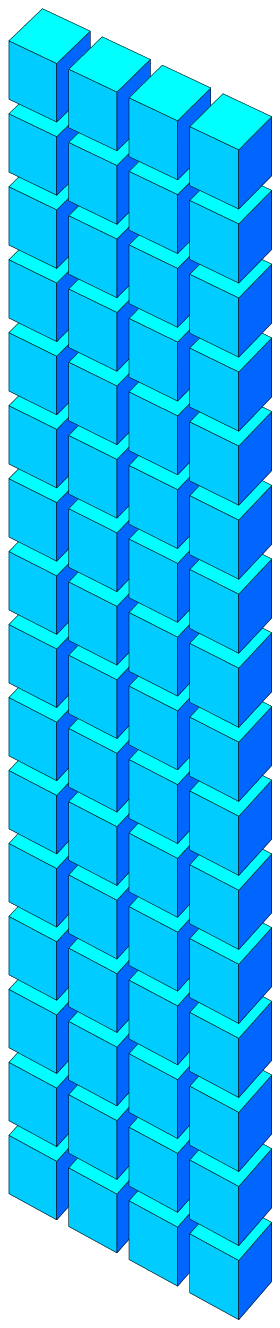}\hfill
\raisebox{-1in}{\includegraphics[width=.4\textwidth]{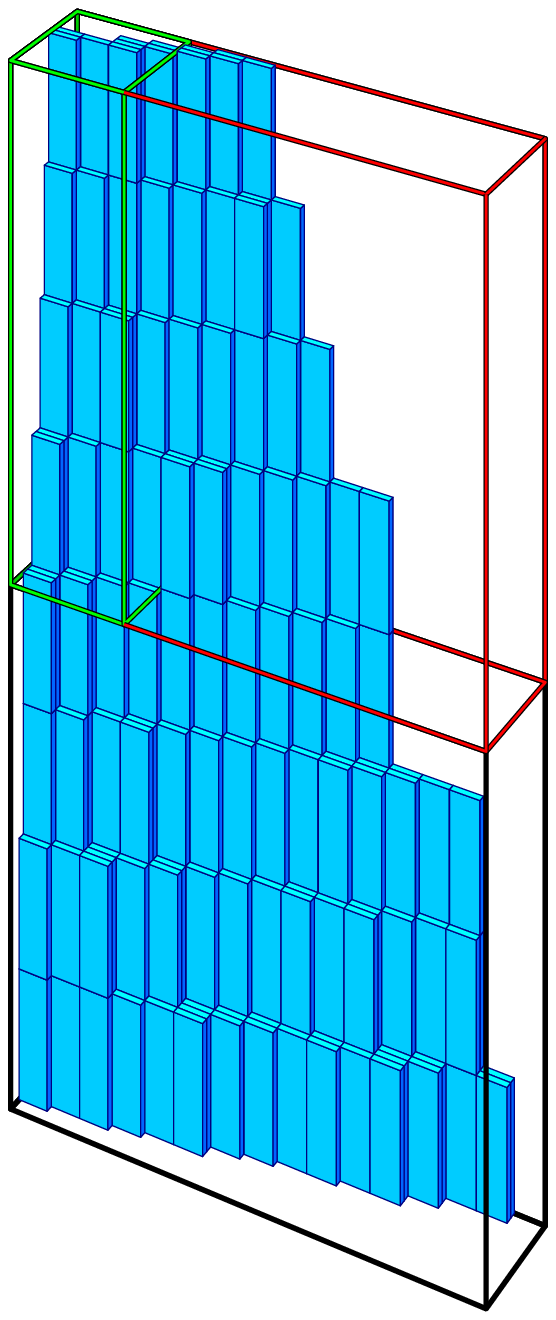}}
\hfill{}
\caption{\emph{Left:} An example of a cell $\cell(x)$, which is a box of dimensions $(1,4,16)$.
\emph{Right:} A good brick $B$ and its sail $S$ for $L=4$.
This is obtained by scaling the set $\hS$ of
\cref{fig-proto} and replacing each of its sites with a cell.
The head of the brick is outlined in red, and its tip in green. Again,
the distinguished corner, the origin, is at the back, hidden by the sail.
}\label{fig-cell}
\end{figure}

We next show that with high probability a brick is good,
and moreover the sail can be chosen to contains a specific
site.

\begin{prop}\label{good-likely}
Assume $L=\lceil p^{-128}\rceil$.  Then the probability
that $B$ is good and has a sail $S$ that contains the site
$x_0:=(L+1,4L+4,16L)$ converges to $1$ as $(p,q)\to (0,0)$.
\end{prop}

We remark that $L$ does not need to be such a large power
of $p^{-1}$; with suitable modifications to the definitions and proofs
(perhaps at the expense of increased complexity), order
$p^{-1}\log p^{-1}$ would suffice.

\begin{proof}[Proof of \cref{good-likely}]
Fix $\epsilon>0$.  For most of the proof we will consider
the relevant event on the proto-brick $\hB$.  Therefore let
$\hp=p^{64}$ and $\hq=1-(1-q)^{64}$, which are the
probabilities that a site is, respectively, initially
occupied and open in the auxiliary configuration.  Let
$L=\lceil p^{-128}\rceil$.

Call a site $x$ \df{swell} if $x$ and $x+(0,0,1)$ and
$x+(-1,-1,1)$ are all open in the auxilliary configuration.
By the results of \cite{LSS}, the
configuration of swell sites dominates a product measure on
$\Z^3$ with parameter $1-\hqp$, where $\hqp=\hqp(\hq)\to 0$
as $\hq\to 0$.

\newcommand{\hD}{\widehat D}

Next, we apply \cref{surface} and translation invariance to
construct a swell curtain close to the half-space
$(L,L,L)+H$, rather than $H$. To be precise, translate the
configuration of swell sites by $-(L,L,L)$, construct the
set $D$ according to the last section, but using swell
sites in place of open sites, and translate it back by
$(L,L,L)$ to obtain a set $\hD$ of swell sites that
 lies in $((L,L,L)+H)^C=\{x:x_1+x_2+x_3>3L\}$.

Let $E_1$ be the event that $\hD$ is a curtain and contains
the site $\hx_0=(1,1,0)+(L,L,L)$.  By the construction of
$\hD$ in the previous section, $E_1$ is an increasing event
with respect to the configuration of swell sites. (This
follows because, in the notation of that section, the set
$A$ of sites reachable from $H$ via permissible paths is
decreasing). Therefore, by \cref{surface} and \cite{LSS},
there exists $q_1>0$ such that if $\hq<q_1$ then
$\P(E_1)>1-\epsilon$.

Moreover, by \cref{surface}, \cite{LSS}, and translation
invariance, for any deterministic $x=(x_1,x_2,x_3)$ with
$x_1+x_2+x_3=3L$ we have
\begin{equation}\label{good-likely-eq1}
\P\Bigl(\hD\cap\bigl\{x+(t,t,0):t\in[1,k]\bigr\}\ne \emptyset\Bigr)\ge 1-e^{-ck},
\qquad k>0,
\end{equation}
where $c>0$ is an absolute constant.  (This event is again
increasing in the configuration of swell sites, by the
construction of $\hD$.)
%
Now let $$\hS:=\hD\cap \hB.$$
Let $E_2$ be the event that every $x\in \hS$ satisfies
$x_1+x_2+x_3<4L$.  Then \cref{good-likely-eq1} and a union
bound imply that $\P(E_2)\geq 1- 16L^2\exp(-cL/2).$ Since
$L\to \infty$ as $p\to 0$ (i.e.\ as $\hp\to 0$), for $\hp$
is sufficiently small we have $\P(E_2)\geq 1-\epsilon$.

We have shown that $\hp$ and $\hq$ are both sufficiently
small then $\P(E_1 \cup E_2)\geq 1-2\epsilon$.  On $E_1\cup
E_2$, the set $\hS$ satisfies properties (G1)--(G4) in the
definition of a good proto-brick.  So far we have not
considered initially occupied sites (although the parameter
$p$ has appeared in the definition of the length scale
$L$).  One way to sample the auxiliary configuration is as
follows. First declare each site closed independently with
probability $\hq$.  Then, conditional on the resulting
configuration, declare each open site to be initially
occupied independently with probability $\hp/(1-\hq)$
($\geq \hp$). Let $E_3$ be the event that $\hS$ satisfies
property (G5). On $E_1\cap E_2$, each intersection with a
layer $\hS\cap\lev_k$ for $k\in[0,2L-1)$ contains at least
$L$ sites (by (G3) and (G4)).  Moreover, all sites in
$(\hS\cap \lev_k)+e_3$ are open (by (G1)). Hence,
$$\P\bigl(E_3\mid E_1\cap E_2\bigr) \geq
1-2L(1-\hp)^L\ge 1-2L\exp(-\hp L).
$$
Since $L\sim \hp\,^{-2}$, this is at least $1-\epsilon$ for
$\hp$ sufficiently small.

We have shown that for $\hp$ and $\hq$ sufficiently small,
with probability at least $1-3\epsilon$ the set $\hS$
satisfies (G1)--(G5) and contains $\hx_0$. Finally,
recalling the definition of the auxilliary configuration,
we deduce that for $p$ and $q$ sufficiently small, the
brick $B$ is likewise good and has a sail containing
$x_0\in\cell(\hx_0)$ with probability at least
$1-3\epsilon$.
\end{proof}

\section{Activation}\label{sec-activation}

Recall that a sail of a good brick is said to be activated
if its head is fully occupied (at some time).  To enable
our renormalization argument, we now show that for
appropriately placed good bricks, activation of one sail
leads to activation of another.

Let $B$ be the brick in standard position as before, and
let $B'$ be a brick with dimensions $(32L, 4L, 16L)$ such
that the centroid of its tail coincides with the centroid
of the tip of $B$.  (The idea is that the tip of $B$ cuts
the tail of $B'$ in two. There are eight possible choices of
$B'$: two possible boxes that share a tail, each with four possible
orientations.  See \cref{fig-bricksm} in the next section
for examples.) Then we write $B\boxright B'$. Similarly for
any isometry $\eta$ of $\Z^3$ we write $\eta(B)\boxright
\eta(B')$.

\begin{prop}\label{triggering}
Let $B$ and $B'$ be as described above. Suppose that they
are both good and let $S$ and $S'$ be any respective sails.
In the modified bootstrap percolation model, if $S$ is
activated by some time, then $S'$ is activated by some
later time.
\end{prop}

We separate the proof into the following four lemmas,
starting with the underlying growth mechanism.

\begin{lemma}\label{growth-sail}
Suppose that the proto-brick $\hB$ is good, and let $\hS$
be any set satisfying the conditions in the definition of
good. Assume also that the intersection $\hS\cap \lev_0 $
with the bottom layer is entirely occupied initially, and
that $\Z^3\setminus \sigma(\hS)$ is entirely closed. Then
$\hS$ is entirely occupied at some time.
\end{lemma}

\begin{proof} The argument is essentially the same as for \cref{curtain-spread},
except that one must verify that the relevant sites lie in
the proto-brick. We prove by induction on $k=0,\ldots,
2L-1$ that the layer $\hS\cap\lev_k$ is eventually
occupied. For $k=0$ this holds by assumption.

Fix $k\geq 1$ and let $Y=(\hS\cap\lev_{k-1})+e_3$.  Then
$Y$ becomes occupied, since it is connected and open, it
contains an occupied site, and it is adjacent to
$\hS\cap\lev_{k-1}$ which becomes occupied by the inductive
hypothesis.  If $z\in \hS\cap\lev_k$ then either $z\in Y$
or $z=y-(1,1,0)$ where $y\in Y$.  In the latter case there
exist $z+a e_1,z+ b e_2\in\pi+e_3\in Y$ with $a,b\in[0,3]$.
The bootstrap rule then guarantees that $z+ie_1$ and
$z+je_2$ become occupied for $i\in(0,a)$ and $j\in(0,b)$,
and then $z$ becomes occupied.
\end{proof}

The following comparison lemma states that cutting off part
of a configuration only increases the eventually occupied
set, provided we make the cut surface occupied. This will
enable us to make use of sails that intersect each other.

\begin{lemma}\label{sneaky}
Consider a set of sites $A$, and a subset $F\subseteq A$.
Let $B$ be a connected component of $A\setminus F$. Suppose
that every site in $A^C$ is closed but that the initial
configuration is otherwise arbitrary.  Now alter the
initial configuration by making $F$ initially occupied but
$A\setminus (F\cup B)$ closed.  The alteration (weakly)
increases the set of eventually occupied sites in $B$.
\end{lemma}

\begin{proof}
We proceed by induction on the time step. Suppose that at
all times prior to $t$, the set of occupied sites of $B$ in
the altered dynamics dominates the set in the original
dynamics.  Assume that a site $x\in B$ becomes occupied in
the original dynamics at time $t$.  Any neighbor of $x$
that was occupied in the original dynamics at time $t-1$
either lies in $B$, in which case it is also occupied in
the altered dynamics by the induction hypothesis, or it
lies in $F$, in which case it was \emph{initially} occupied
in the altered dynamics. Thus $x$ also becomes occupied in
the altered dynamics.
\end{proof}

\begin{lemma} \label{stretching}  From any configuration on $B$,
form the auxiliary configuration on $\hB$, and perform the modified bootstrap percolation
dynamics from the auxiliary configuration with all sites
outside $\hB$ closed. If $x$ becomes occupied in the
auxiliary dynamics, then $\cell(x)$ becomes fully occupied
in the original dynamics.
\end{lemma}

\begin{proof}
This follows by straightforward induction on time step.
\end{proof}

Next we state a geometric fact about sails. Let
$A\subseteq\Z^3$ and let $F,B_1,B_2$ be disjoint subsets of
$A$. We say that $F$ separates $B_1$ and $B_2$ in $A$ if
$A\setminus F$ contains no nearest-neighbor path from $B_1$
to $B_2$.

\begin{lemma}\label{separation} Suppose that the brick $B$ is good. Then any
sail $S$ for $B$ separates, in the tip of $B$, the two
faces of the tip $\{0\}\times [0, 4L) \times [16L,32L)$ and
$\{4L-1\}\times [0, 4L) \times [16L,32L)$.
\end{lemma}

\begin{proof}
By property (G4) of a good proto-brick, the intersection of
$S$ with a layer $\lev_k$ is an oriented path, thickened by
conversion of sites to cells. Therefore its complement
$\lev_k\setminus S$ clearly has two components, $U_k$ and
$V_k$ say, which contain the intersections of the first and
second faces respectively with $\lev_k$, by (G3).

It remains to check that no site of $U_{k-1}$ is adjacent
to a site of $V_k$, and likewise for $V_{k-1}$ and $U_k$,
for $k=1,\ldots, 4L-1$.  Since such adjacent sites would
differ by $e_3$, this is easily verified from property
(G2).  (Also see \cref{paths}).
\end{proof}

Now we prove the main result of this section.

\begin{proof}[Proof of \cref{triggering}]
By definition of $S$ from $\hS$ and \cref{separation},
the head of $S$ separates the base
of $S'$ from the head of $S'$ in
$\sigma(S')$.

Consider the dynamics from the following new configuration.
Make every site outside $\sigma(S')$ closed. Make the
base of $S'$ occupied. Otherwise, retain the initial
configuration in $\sigma(S')$.
By \cref{stretching,growth-sail},
the entire sail $S'$ becomes occupied.
The proof is concluded by applying \cref{sneaky} to $\sigma(S')$.
\end{proof}

\section{Renormalization}\label{sec-renormalization}

In this section we prove the main result, \cref{main}, as well as \cref{main-follow}.  We
start with a simple geometric ingredient. Recall that $B$
is the brick in standard position.

\begin{samepage}
\begin{lemma} \label{renormalization}
There exist bricks $B_i$, $B_i'$, for $i=1,2,3$, with
\begin{align*}
&B\boxright B_1\boxright B_2\boxright B_3,\\
&B\boxright B_1'\boxright B_2'\boxright B_3',
\end{align*}
such that $B$, $B_3$, and $B_3'$ are distinct and have the
same orientation. Furthermore, there exist vectors
$u,u'\in\Z^3$ and a constant $C$, none of them depending on
$L$, such that $B_3=B+Lu$ and $B_3'=B+Lu'$  (so in particular $Lu$ and $Lu'$
are the distinguished corners of $B_3$ and $B_3'$ respectively),
and all seven bricks lie within distance $CL$ of the origin.
\end{lemma}
\end{samepage}
\begin{figure}
\centering
\includegraphics[width=.36\textwidth]{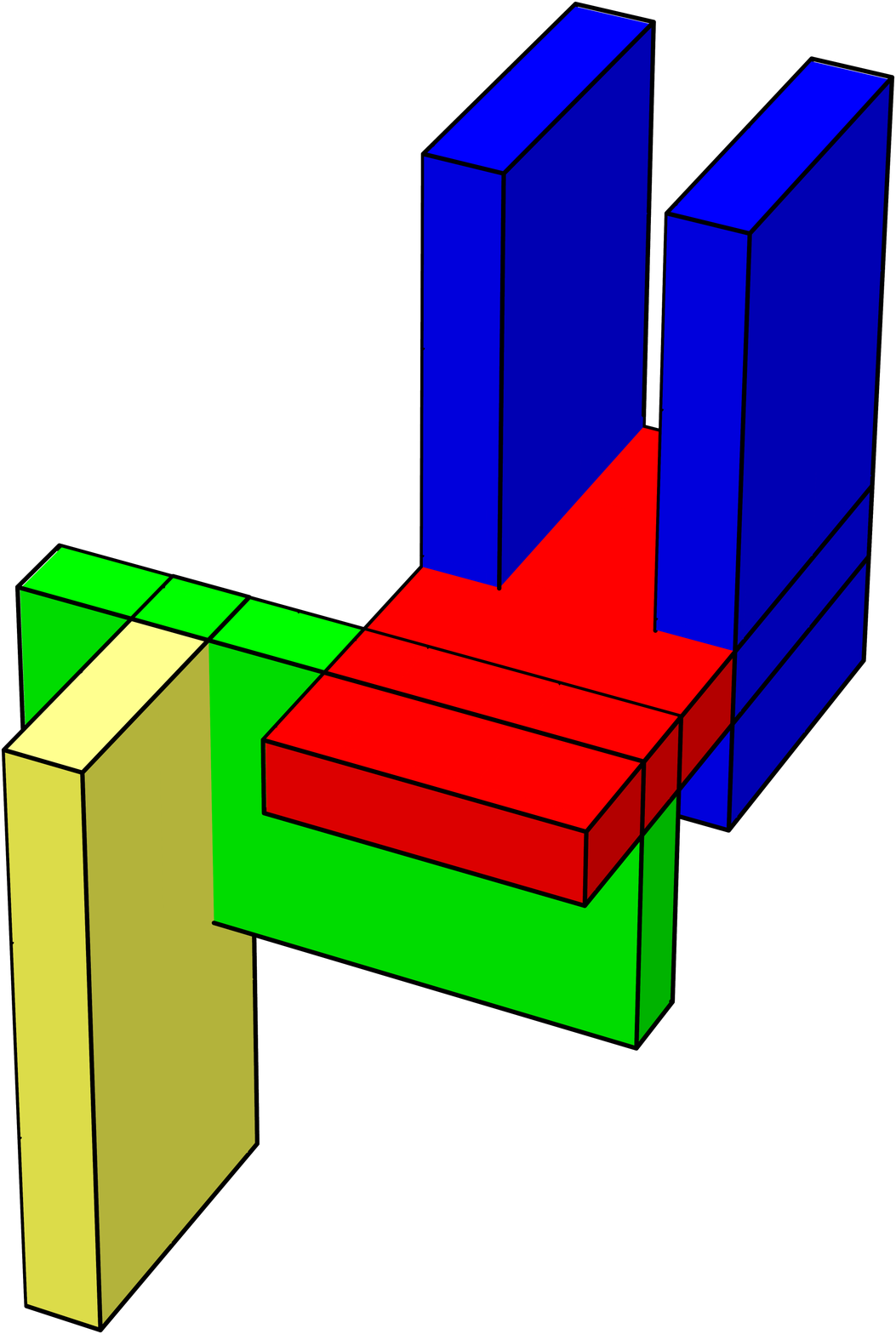}
\caption{\label{fig-bricksm} {Illustration
of the proof of \cref{renormalization}.}
The initial brick $B$ (at lower left) is yellow, and $B_1=B_1'$
is green. The red box depicts $B_2$ and $B_2'$, which
are the same box but with different orientations. Finally, the two blue bricks are
$B_3$ and $B_3'$. These are translations
of $B$ by $(10,22,22)L$ and $(22,22,22)L$.}
\end{figure}

\begin{proof}  See \cref{fig-bricksm}.
Recall that $B$ has dimensions $(4L,16L,32L)$.  We choose
$B_1$ and $B_1'$ equal to each other, with dimensions
$(32L,4L,16L)$, and satisfying $B\rhd B_1$. Then take $B_2$
and $B_2'$ to be the same box as each other, with
dimensions $(16L,32L,4L)$, but with different orientations
and in particular different tips.  Finally take $B_3$ and $B_3'$ to be
suitable translations of $B$, as determined by these tips.
\end{proof}

\newcommand{\Pper}{\P_{\text{\tt per}}}
\newcommand{\Pnuc}{\P_{\text{\tt nuc}}}

\begin{proof}[Proof of \cref{main}]
As discussed in the introduction, it suffices to prove the case of the
modified bootstrap model on $\Z^3$.

We will compare with oriented
percolation in $\Z^2$. Let $L=\lceil p^{-128}\rceil$ and
let $u,u'\in\Z^3$ be as in \cref{renormalization}. Also fix
$\epsilon>0$. For $a=(a_1,a_2)\in\Z^2$, define the
associated brick
$$B(a):=B+L a_1 u+ La_2 u'.$$
Call the site $a$ \df{excellent} if the translations by
$La_1 u+ La_2 u'$ of the seven bricks $B,B_i,B_i'$ of
\cref{renormalization} are all good.

Suppose that $a$ and $b$ are excellent, so that in
particular the bricks $B(a)$ and $B(b)$ are good.  Suppose
also that there is a path of excellent sites in $\Z^2$ from
$a$ to $b$ consisting of steps $e_1$ and $e_2$. (We call a
path with these steps \df{oriented}). Then by
\cref{renormalization,triggering}, if some sail of $B(a)$
is activated then any sail of $B(b)$ is activated at some
later time.

Let $E$ be the event that there exists an
 excellent bi-infinite oriented path $\pi$ in $\Z^2$ containing $0=(0,0)$, and
that moreover $B=B(0)$ has a good sail containing
$x_0:=(L+1,4L+4,16L)$ in its head. By
\cref{renormalization}, the random configuration of
excellent sites is $m$-dependent for some fixed $m$ not
depending on $L$.  Therefore, by \cite{LSS},
\cref{good-likely}, and the fact that oriented percolation
on $\Z^2$ has a nontrivial phase transition
 (see e.g.~\cite{g}), if $p$ and $q$ are
sufficiently small then $\P(E)\geq 1-\epsilon$.

It remains to show that \emph{some} sail on the path is
activated, for which a rather crude sprinkling argument
will suffice. Assuming $2p+q<1$, we consider two coupled
initial configurations.  The \df{level-$1$} configuration
has parameters $p$ and $q$ as before.  Conditional on the
level-$1$ configuration, the \df{level-$2$} configuration
is obtained by adding some further occupied sites;
specifically, we declare each open site that was not
initially occupied at level $1$ to be initially occupied at
level $2$ independently with probability $p/(1-p-q)$, and
leave the configuration otherwise unchanged. The law of the
level-$2$ configuration is simply a product measure with
parameters $2p$ and $q$.  Now condition on the level-$1$
configuration, and suppose that it is such that $E$ occurs
at level $1$.  Fix an excellent oriented path $\pi$ as in
the definition of $E$, and let $\pi^+$ and $\pi^-$ be the
forward and backward halves of $\pi$ that start at $0$ and
end at $0$ respectively. Then for each site $a$ of $\pi^-$,
\emph{all} open sites in the brick $B(a)$ are initially
occupied at level $2$ with probability at least $p^{|B|}$,
independently for each such $a$. Therefore, conditionally
almost surely, some site $a$ on $\pi^-$ has this property,
which implies in particular that any sail of the associated
brick $B(a)$ is activated at level $2$.

We conclude that if $2p$ and $q$ are sufficiently small
then with probability at least $1-\epsilon$ there exists an
infinite sequence of distinct activated sails, each
intersecting the next, one of which contains $x_0$ in its
head. By translation invariance we conclude that with
probability at least $1-\epsilon$, the origin lies in an
infinite connected eventually occupied set, as required.
\end{proof}

\begin{proof}[Proof of \cref{main-follow}]
It follows
from \cref{main} that $\lim_{q\to 0^+} \phi^+(q)=1$,
which implies that $q_c>0$.
As $\phi$ is a decreasing function, it is
positive on $[0,q_c)$. Our remaining task is to prove the claimed
upper bound on $q_c$, for which it suffices to consider the \emph{standard} (as opposed to modified) bootstrap model with $r=2$ on $\Z^d$ for $d\geq 3$.

\newcommand{\cC}{\mathcal Z}
Call a site \df{$3$-open} if it is open and has at least
$3$ open sites among its $2d$ neighbors. Let $p_c'$ be the
critical probability for existence of an infinite connected
set of $3$-open sites in $\Z^d$. Then clearly $p_c'\geq
p_c^{\text{\rm site}}$. For $d=3$, the method of essential
enhancements \cite{AG,BBR} shows that $p_c'> p_c^{\text{\rm
site}}$. (The strict inequality is expected to hold for
$d\geq 4$ also, but no complete proof is available -- see
\cite{BBR}).

For any set $Z\subseteq\Z^d$, let the external boundary $\partial Z$ be the set of sites
in $\Z^d\setminus Z$ that have a neighbor in $Z$.  Note that $|\partial Z|\le 2d|Z|$.
Assume that no
site in $\partial Z$ is $3$-open, and that no site in $Z\cup\partial Z$ is initially occupied. Then we claim that no site
in $Z\cup\partial Z$ is ever occupied. Indeed, suppose on the contrary that
$x\in Z\cup \partial Z$ is a first site in the
set to become occupied, say at time $t\ge 1$. Then $x\in \partial Z$,
and so $x$ has
at most $2$ open neighbors, of which at least one is in $Z$,
which by assumption
is unoccupied at time $t-1$.  So $x$ has at most one occupied neighbor at time $t-1$. Since $r=2$, this contradicts the assumption that $x$ becomes occupied at time
$t$.

Now let $q>1-p_c'$.  Given the random configuration on $\Z^d$,
create an \df{adjusted} configuration by converting all closed sites among the origin and its $2d$ neighbors to open (but not initially occupied) sites.
Let $\cC$ be the maximal connected set of $3$-open sites
containing the origin in the adjusted configuration.
Clearly, $0<|\cC|<\infty$ almost surely. Then we have
\begin{equation*}
\begin{aligned}
&\P(0\text{ is eventually occupied})\\
&\le \P(0\text{ is eventually occupied starting from the adjusted configuration})\\
&\le \P(\text{$\cC\cup\partial\cC$ contains an initially occupied site})\\
&\le\sum_{k=1}^\infty \P\bigl(|\cC|=k\bigr) \bigl(1-(1-p)^{k+2dk}\bigr).
\end{aligned}
\end{equation*}
This tends to $0$ as $p\to 0^+$, by dominated convergence.
\end{proof}

\section{Open Problems} \label{sec-open}

Recall that $\phi(p,q)$ is the probability that the origin is eventually occupied
with densities $p$ and $q$ of initially occupied and closed sites respectively, and that we define $\phi^+(q)=\phi(0^+,q)$ and $q_c=\inf\{q: \phi^+(q)=0\}$.

\begin{enumerate}[(i)]
\item  For which dimensions and thresholds $d> r\geq 3$ is it the case that $\phi(p,q)\to 1$ as $(p,q)\to (0,0)$?
    As conjectured in \cite{Mor}, the answer ``all'' seems plausible.
    (The current paper proves the cases $d>r=2$, while the conclusion fails for $d=r$.)
\item Where the convergence in (i) above does not hold
    (presumably, only for $d=r$), suppose that $p,q\to 0$
    in such a way that $ \log q/\log p\to\alpha. $ For
    which $\alpha$ does $\phi$ converge to $0$, or to
    $1$?  The articles \cite{GM} and \cite{GHS} address
    $d=r=2$ and $d=r=3$ respectively.

\item Is $\phi^+$ continuous at $q_c$?

\item Consider the critical value $q_c=q_c(d)$ as a function of dimension (with $r=2$, say).
Does $q_c$ approach $1$ as $d\to\infty$ and, if so, at what rate?


\item Let $T$ be the first time the origin is occupied.
What is the asymptotic behavior of $T$ as $p,q\to 0$?  (For example, find ``close'' functions $f$ and $g$ of $p$ and $q$ for which $f\leq T\leq g$ with high probability).

\item For $r=2$ and $d=3$, consider
$$
\gamma(q):=\limsup_{p\to 0+}p^{-1}\phi(p,q).
$$
Is there a $q>q_c$ for which this is infinite?  If so, this would
distinguish the phase transition in the case $r=2$ from that of the case $r=1$ (where $q_c=1-p_c^{\text{\rm site}}(\Z^d)$, and $\gamma(q)$ is finite for all $q>q_c$.)
\end{enumerate}

\section*{Acknowledgements}

We thank David Sivakoff for many valuable discussions.  
Janko Gravner was partially supported by the NSF grant DMS--1513340, Simons Foundation Award \#281309, and the Republic of Slovenia’s Ministry of Science program P1--285.  He also gratefully acknowledges the hospitality of the Theory Group at Microsoft Research, where most of this work was completed.

\bibliographystyle{alphanum}
\bibliography{pbp}

\newcommand{\etalchar}[1]{$^{#1}$}
\begin{thebibliography}{BDGM2}

\bibitem[AFP]{amini2}
H.~Amini, N.~Fountoulakis, and K.~Panagiotou.
\newblock Bootstrap percolation in inhomogeneous random graphs.
\newblock arXiv:1402.2815.

\bibitem[AG]{AG}
M.~Aizenman and G.~Grimmett.
\newblock Strict monotonicity for critical points in percolation and
  ferromagnetic models.
\newblock {\em J. Statist. Phys.}, 63(5-6):817--835, 1991.

\bibitem[AL]{AL}
M.~Aizenman and J.~L. Lebowitz.
\newblock Metastability effects in bootstrap percolation.
\newblock {\em J. Phys. A}, 21(19):3801--3813, 1988.

\bibitem[BBDM]{BBDM}
J.~Balogh, B.~Bollob\'as, H.~Duminil{-}Copin, and R.~Morris.
\newblock The sharp threshold for bootstrap percolation in all dimensions.
\newblock {\em Trans. Amer. Math. Soc.}, 364(5):2667--2701, 2012.

\bibitem[BBR]{BBR}
B.~Bollob\'as, P.~Balister, and O.~Riordan.
\newblock Essential enhancements revisited.
\newblock arXiv:1402.0834.

\bibitem[BDGM1]{BDGM2}
G.~J. Baxter, S.~N Dorogovtsev, A.~V Goltsev, and J.~Mendes.
\newblock Heterogeneous k-core versus bootstrap percolation on complex
  networks.
\newblock {\em Physical Review E}, 83(5):051134, 2011.

\bibitem[BDGM2]{BDGM1}
G.~J. Baxter, S.~N. Dorogovtsev, A.~V. Goltsev, and J.~F.~F. Mendes.
\newblock Bootstrap percolation on complex networks.
\newblock {\em Phys. Rev. E}, 82:011103, Jul 2010.

\bibitem[BDMS]{BDMS}
B.~Bollob\'as, H.~Duminil{-}Copin, R.~Morris, and P.~Smith.
\newblock Universality of two-dimensional critical cellular automata.
\newblock {\em Proceedings of London Mathematical Society}.
\newblock To appear.

\bibitem[BT]{bod-tex}
T.~Bodineau and A.~Teixeira.
\newblock Interface motion in random media.
\newblock {\em Comm. Math. Phys.}, 334(2):843--865, 2015.

\bibitem[CLR]{CLR}
J.~Chalupa, P.~L. Leath, and G.~R. Reich.
\newblock Bootstrap percolation on a {B}ethe lattice.
\newblock {\em Journal of Physics C: Solid State Physics}, 12(1):L31, 1979.

\bibitem[DDG{\etalchar{+}}]{DDGHS}
N.~Dirr, P.~W. Dondl, G.~R. Grimmett, A.~E. Holroyd, and M.~Scheutzow.
\newblock Lipschitz percolation.
\newblock {\em Electron. Commun. Probab.}, 15:14--21, 2010.

\bibitem[DDS]{interface}
N.~Dirr, P.~W. Dondl, and M.~Scheutzow.
\newblock Pinning of interfaces in random media.
\newblock {\em Interfaces Free Bound.}, 13(3):411--421, 2011.

\bibitem[DEK{\etalchar{+}}]{DEKNS}
M.~Damron, S.~M. Eckner, H.~Kogan, C.~M. Newman, and V.~Sidoravicius.
\newblock Coarsening dynamics on {$\mathbb{Z}^d$} with frozen vertices.
\newblock {\em J. Stat. Phys.}, 160(1):60--72, 2015.

\bibitem[DvE]{DE}
H.~Duminil{-}Copin and A.~C.~D. van Enter.
\newblock Sharp metastability threshold for an anisotropic bootstrap
  percolation model.
\newblock {\em Ann. Probab.}, 41(3A):1218--1242, 2013.

\bibitem[GG]{GG}
J.~Gravner and D.~Griffeath.
\newblock First passage times for threshold growth dynamics on {${\bf Z}^2$}.
\newblock {\em Ann. Probab.}, 24(4):1752--1778, 1996.

\bibitem[GH1]{ent}
G.~R. Grimmett and A.~E. Holroyd.
\newblock Plaquettes, spheres, and entanglement.
\newblock {\em Electron. J. Probab.}, 15:1415--1428, 2010.

\bibitem[GH2]{geom}
G.~R. Grimmett and A.~E. Holroyd.
\newblock Geometry of {L}ipschitz percolation.
\newblock {\em Ann. Inst. Henri Poincar\'e Probab. Stat.}, 48(2):309--326,
  2012.

\bibitem[GH3]{embed}
G.~R. Grimmett and A.~E. Holroyd.
\newblock Lattice embeddings in percolation.
\newblock {\em Ann. Probab.}, 40(1):146--161, 2012.

\bibitem[GHM]{GHM}
J.~Gravner, A.~E. Holroyd, and R.~Morris.
\newblock A sharper threshold for bootstrap percolation in two dimensions.
\newblock {\em Probab. Theory Related Fields}, 153(1-2):1--23, 2012.

\bibitem[GHS]{GHS}
J.~Gravner, A.~E. Holroyd, and D.~Sivakoff.
\newblock Polluted bootstrap percolation in three dimensions.
\newblock In preparation.

\bibitem[GM1]{GMa}
O.~Garet and R.~Marchand.
\newblock Asymptotic shape for the chemical distance and first-passage
  percolation on the infinite {B}ernoulli cluster.
\newblock {\em ESAIM Probab. Stat.}, 8:169--199, 2004.

\bibitem[GM2]{GM}
J.~Gravner and E.~McDonald.
\newblock Bootstrap percolation in a polluted environment.
\newblock {\em J. Statist. Phys.}, 87(3-4):915--927, 1997.

\bibitem[Gri]{g}
G.~R. Grimmett.
\newblock {\em Percolation}.
\newblock Springer-Verlag, Berlin, second edition, 1999.

\bibitem[GZH]{GZH}
J.~Gao, T.~Zhou, and Y.~Hu.
\newblock Bootstrap percolation on spatial networks.
\newblock {\em Scientific reports}, 5, 2015.

\bibitem[HM]{comb}
A.~E. Holroyd and J.~B. Martin.
\newblock Stochastic domination and comb percolation.
\newblock {\em Electron. J. Probab.}, 19:no. 5, 16, 2014.

\bibitem[Hol]{Hol1}
A.~E. Holroyd.
\newblock Sharp metastability threshold for two-dimensional bootstrap
  percolation.
\newblock {\em Probab. Theory Related Fields}, 125(2):195--224, 2003.

\bibitem[J{\L}TV]{JLTV}
S.~Janson, T.~{\L}uczak, T.~Turova, and T.~Vallier.
\newblock Bootstrap percolation on the random graph {$G_{n,p}$}.
\newblock {\em Ann. Appl. Probab.}, 22(5):1989--2047, 2012.

\bibitem[KL]{koch}
C.~Koch and J.~Lengler.
\newblock Bootstrap percolation on geometric inhomogeneous random graphs.
\newblock arXiv:1603.02057.

\bibitem[LSS]{LSS}
T.~M. Liggett, R.~H. Schonmann, and A.~M. Stacey.
\newblock Domination by product measures.
\newblock {\em Ann. Probab.}, 25(1):71--95, 1997.

\bibitem[Mor]{Mor}
R.~Morris.
\newblock Bootstrap percolation, and other automata.
\newblock {\em European Journal of Combinatorics}.
\newblock To appear.

\bibitem[MP]{evolve}
B.~Morris and Y.~Peres.
\newblock Evolving sets and mixing.
\newblock In {\em Proceedings of the {T}hirty-{F}ifth {A}nnual {ACM}
  {S}ymposium on {T}heory of {C}omputing}, pages 279--286. ACM, New York, 2003.

\bibitem[Sch]{schonmann}
R.~H. Schonmann.
\newblock On the behavior of some cellular automata related to bootstrap
  percolation.
\newblock {\em Ann. Probab.}, 20(1):174--193, 1992.

\bibitem[vE]{van-enter}
A.~C.~D. van Enter.
\newblock Proof of {S}traley's argument for bootstrap percolation.
\newblock {\em J. Statist. Phys.}, 48(3-4):943--945, 1987.

\end{thebibliography}

\end{document}